\documentclass[11pt]{article}
\usepackage{latexsym,amsmath,amssymb,amsfonts,mathrsfs,amsthm}
\usepackage{epsf,graphicx,epsfig,color,cite,cases}
\usepackage{subfigure,graphics,multirow,marginnote,enumerate,bm}
\usepackage{indentfirst}
\usepackage{mathtools}
\usepackage{bm}
\numberwithin{equation}{section}

\usepackage{enumerate}

\usepackage{tikz}

\newcommand{\R}{{\mat R}}

\newcommand{\C}{{\mat C}}

\newcommand{\be}{\begin{eqnarray}}
\newcommand{\ben}{\begin{eqnarray*}}
\newcommand{\en}{\end{eqnarray}}
\newcommand{\enn}{\end{eqnarray*}}

\newcommand{\mat}{\mathbb}

\newtheorem{theorem}{Theorem}[section]
\newtheorem{lemma}[theorem]{Lemma}

\newtheorem{remark}[theorem]{Remark}

\DeclareMathOperator{\dif}{d\!}
\DeclareMathOperator{\supp}{\mathrm{supp}}
\renewcommand{\Re}{\operatorname{Re}}

\begin{document}
\renewcommand{\theequation}{\arabic{section}.\arabic{equation}}

\title{\bf Uniqueness to inverse acoustic and elastic medium scattering problems with hyper-singular source method}

\author{\bf Chun Liu\thanks{School of Mathematical Sciences and LPMC, Nankai University, Tianjin, 300071, China (liuchun@nankai.edu.cn, ghhu@nankai.edu.cn)},
Guanghui Hu$^*$,
Jianli Xiang\thanks{Three Gorges Mathematical Research Center, College of Science, China Three Gorges University, Yichang, 443002, China (xiangjianli@ctgu.edu.cn)}
and Jiayi Zhang\thanks{Corresponding author. School of Mathematical Sciences and LPMC, Nankai University, Tianjin, 300071, China (zhangjy97@mail.nankai.edu.cn)}}

\date{}

\maketitle


\begin{abstract}
  This paper is concerned with inverse scattering problems of determining the support of an isotropic and homogeneous penetrable body from knowledge of multi-static far-field patterns in acoustics and in linear elasticity. The normal derivative of the total fields admits no jump on the interface of the scatterer in the trace sense. If the contrast function of the refractive index function or the density function has a positive lower bound near the boundary, we propose a hyper-singular source method to prove uniqueness of inverse scattering with all incoming plane waves at a fixed energy. It is based on subtle analysis on the leading part of the scattered field when hyper-singular sources caused by the first derivative of the fundamental solution approach to a boundary point. As a by-product, we show that this hyper-singular method can be also used to determine the boundary value of a H\"older continuous refractive index function in acoustics or a H\"older continuous density function in linear elasticity.
  \vspace{.2in}


\end{abstract}

\maketitle
\section{Introduction}
\subsection{Acoustic medium scattering problem}
Consider a time-harmonic acoustic plane wave incident onto a bounded penetrable obstacle $D \subset \mathbb{R}^{3}$ embedded in a homogeneous isotropic medium. The spatially-dependent incident field $u^{{\rm in}}(x,d)$ takes the form $u^{{\rm in}}(x,d) = e^{ikx \cdot d}, \ x \in \mathbb{R}^{3}$, where $d \in \mathbb{S}^{2} \coloneqq\{x \in \mathbb{R}^{3} \colon \vert x \vert = 1\}$ is the incident direction and $k > 0$ is the wave number. We assume that the complement of $D$, $D^{e} \coloneqq \mathbb{R}^{3} \backslash \bar{D}$, is connected and $D$ is of $C^2$-smoothness. The acoustic properties of the scatterer can be described by the refractive index function $n(x)$ such that $n(x)=1$ in $D^{e}$ after some normalization. Hence the contrast function $1 - n(x)$ is compactly supported in $D$. The wave propagation is modeled by the Helmholtz equation
\begin{equation} \label{eq:helmholtz}
\Delta u + k^{2}\,n\,u = 0 \text{ in } \mathbb{R}^{3}.
\end{equation}
In (\ref{eq:helmholtz}), $u = u^{{\rm in}} + u^{{\rm sc}}$ denotes the total field, where
$u^{{\rm sc}}$ is the perturbed scattered field satisfying the Sommerfeld radiation
condition
\begin{equation} \label{eq:sommerfeld}
\lim_{\vert x\vert \to \infty} r \bigg\{\frac{\partial u^{{\rm sc}}}{\partial r} - ik u^{{\rm sc}}\bigg\} = 0,\ r = \vert x\vert.	
\end{equation}
The Sommerfeld radiation condition (\ref{eq:sommerfeld}) leads to the asymptotic expansion
\begin{equation*} \label{eq:farfield}
u^{{\rm sc}}(x) = \frac{e^{ik\vert x\vert}}{\vert x\vert}u^{\infty}(\hat{x}) + \mathcal{O}\left(\frac{1}{\vert x\vert^{3/2}}\right),\ \vert x\vert \to +\infty,
\end{equation*}
uniformly in all directions $\hat{x} \coloneqq x/\vert x\vert$, $x \in \mathbb{R}^{3}$. The function $u^{\infty}(\hat{x})$ is an analytic function defined on $\mathbb{S}^{2}$ and is referred to as the far-field pattern or the scattering amplitude. The vector $\hat{x} \in \mathbb{S}^{2}$ is called the observation direction of the far field. We also need transmission conditions on the interface $\partial D$. In this paper we assume the {\em continuity} of the total field and its normal derivative in the trace sense that,
\begin{equation} \label{eq:boundarycondition}
	u^{+} = u^{-},\ \partial_{\nu} u^{+} = \partial_{\nu} u^{-} \text{ on } \partial D.
\end{equation}
Here $\nu \in \mathbb{S}^{2}$ is the unit normal vector on $\partial D$ pointing into $D^{e}$, and the superscripts $(\cdot)^{\pm}$ stand for the limits taken from outside and inside, respectively.

The direct scattering problem for (\ref{eq:helmholtz}), (\ref{eq:sommerfeld}) and (\ref{eq:boundarycondition}) can be formulated as a task of finding the scattered field $u^{{\rm sc}}$ when the incident field $u^{{\rm in}}$ and $n$ are given. If the refractive index $n\in C(\overline{D})$ satisfies $\Re n>0$ and $\Im n\geq0$, it was shown in \cite[Theorems 8.1 and 8.3]{colton.kress} that the direct scattering problem admits a unique solution $u^{sc}\in H_{{\rm loc}}^{2}(\mathbb{R}^{3}).$



\subsection{Elastic medium scattering problem}
We also consider an isotropic and inhomogeneous elastic medium in three dimensions. The medium can be described by two Lam\'e coefficients $\lambda$ and $\mu$, and the mass density $\rho$. We simplify the discussion further by supposing that $\lambda$ and $\mu$ are constants satisfying $\mu > 0$ and $2\mu+ 3\lambda > 0$. We also assume that the density $\rho(x)\in L^{\infty}(\mathbb{R}^{3})$ satisfies $\rho(x)=1$ in $D^c:= \mathbb{R}^3\backslash \overline{D}$. Then the vector field $\mathbf{u}=(u_1,u_2,u_3)$ is governed by the following Lam\'e system
\begin{equation}\label{LS}
\Delta^*\mathbf{u}+\omega^2\rho(x)\mathbf{u}=0 \ \mathrm{in} \  \mathbb{R}^3,\ \ \Delta^*\mathbf{u}:=\mu\Delta \mathbf{u}+(\lambda+\mu)\mathrm{grad}\ \mathrm{div}\ \mathbf{u}.
\end{equation}
Here $\omega>0$ is a fixed frequency. Moreover, we assume the exterior $D^c$ is connected. $\mathbf{u}=\mathbf{u}^{{\rm in}}+\mathbf{u}^{{\rm sc}}$ is the total field and $\mathbf{u}^{{\rm sc}}$ is the scattered field. $\mathbf{u}^{{\rm sc}}$ in $D^c$ can be decomposed into the sum of the compressional (longitudinal) part $\mathbf{u}_p^{{\rm sc}}$ and shear (transversal) part $\mathbf{u}_s^{{\rm sc}}$ as follows
\begin{equation}\label{ups}
\mathbf{u}^{{\rm sc}}=\mathbf{u}_p^{{\rm sc}}+\mathbf{u}_s^{{\rm sc}},\ \mathbf{u}_p^{{\rm sc}}=\frac{1}{k_p^2}\mathrm{grad}\ \mathrm{div}\ \mathbf{u}^{{\rm sc}},\ \ \mathbf{u}_s^{{\rm sc}}=\frac{1}{k_s^2}\mathrm{curl}\ \mathrm{curl}\ \mathbf{u}^{{\rm sc}},
\end{equation}
in which
\begin{equation}
\nonumber
k_s:=\omega\sqrt{\frac{\rho}{\mu}},\ \ k_p:=\omega\sqrt{\frac{\rho}{\lambda+2\mu}}
\end{equation}
denote the shear and compressional wave numbers, respectively. It follows from the decomposition in (\ref{ups}) that $\mathbf{u}_p^{{\rm sc}}$ and $\mathbf{u}_s^{{\rm sc}}$ satisfy the Helmholtz equations
\begin{equation}
\nonumber
 (\Delta+k^2_{\alpha})\mathbf{u}_{\alpha}^{{\rm sc}}=0 \ \mathrm{in}\ D^c,\ \ \alpha=p,s.
\end{equation}
The scattered field is required to satisfy the Kupradze radiation condition(see e.g.\cite{Kupradze})
\begin{equation}\label{rd}
\lim\limits_{r\rightarrow \infty}r\left(\frac{\partial \mathbf{u}_{\alpha}^{{\rm sc}}}{\partial r}- ik_{\alpha}\mathbf{u}_{\alpha}^{{\rm sc}}\right)=0, \ \ r=|x|,\ \alpha=p,s,
\end{equation}
uniformly with respect to all directions $\hat{x}=x/|x|\in \mathbb{S}^2:=\{x \in \mathbb{R}^3: |x|=1\}$. The radiation conditions in (\ref{rd}) lead to the P-part $u_p^{\infty}$ and S-part $u_s^{\infty}$ of the far-field pattern of $u^{{\rm sc}}$, given by the asymptotic behavior
\begin{equation*}\label{far}
\mathbf{u}_s^{{\rm sc}}(x)=\frac{\mathrm{exp}(ik_p |x|)}{|x|}\mathbf{u}_p^{\infty}(\hat{x})+\frac{\mathrm{exp}(ik_s |x|)}{|x|}\mathbf{u}_s^{\infty}(\hat{x})+\mathcal{O}\left(\frac{1}{|x|^2}\right),\ |x|\rightarrow \infty,
\end{equation*}
where $\mathbf{u}_p^{\infty}$ and $\mathbf{u}_s^{\infty}$ are the far-field part of $\mathbf{u}_p^{{\rm sc}}$ and $\mathbf{u}_s^{{\rm sc}}$, respectively. In this paper, we define the far-field pattern $\mathbf{u}^{\infty}$ of the scattered field $u^{{\rm sc}}$ as the sum of $\mathbf{u}_p^{\infty}$ and $\mathbf{u}_s^{\infty}$, which means $\mathbf{u}^{\infty}:= \mathbf{u}_p^{\infty}+\mathbf{u}_s^{\infty} $. The transmission conditions on the scatterer's surface $\partial D$ are supposed to be
\begin{equation}\label{boundary}
\mathbf{u}^{+}=\mathbf{u}^{-},\ T\mathbf{u}^{+}= T\mathbf{u}^{-},\ \text{ on } \partial D,
\end{equation}
where
\begin{equation}
T\mathbf{u}:=2\mu \nu\cdot\nabla \mathbf{u}+\lambda \nu \nabla\cdot \mathbf{u}+\mu\nu\times(\nabla\times \mathbf{u})
\end{equation}
is the surface stress operator in terms of the outwardly directed unit normal vector $\nu$. Again the superscripts $(\cdot)^{\pm}$ stand for the limits taken from outside and inside, respectively. The incident wave $\mathbf{u}^{{\rm in}}$ is supposed to an elastic plane wave of the form
\begin{equation*}\label{plane}
\mathbf{u}^{{\rm in}}(x,d,q)=\mathbf{u}_p^{{\rm in}}(x,d)+\mathbf{u}_s^{{\rm in}}(x,d,q)
\end{equation*}
where
\begin{equation*}
\mathbf{u}_p^{{\rm in}}(x,d)=d\ \exp(ik_px\cdot d),\ \  \mathbf{u}_s^{{\rm in}}(x,d,q)=q\ \exp(ik_sx\cdot d)
\end{equation*}
with $d$ being the unit incident direction vector and $q$ being the unit polarization vector satisfying $d\cdot q=0 $.

The direct scattering problem for (\ref{LS})-(\ref{rd}) is to find the scattered field $\mathbf{u}^{{\rm sc}}$ when the incident field $\mathbf{u}^{{\rm in}}$ and density $\rho$ are given. For a real valued $\rho \in \mathit{C}^{1,\gamma}( \mathbb{R}^3)$ with $\rm{supp} (1-\rho)=\overline{D}$, it has been shown in \cite[Theorem 5.10]{peter1998On} that (\ref{LS})--(\ref{boundary}) admits a unique solution $\mathbf{u}^{{\rm sc}}, \mathbf{u}\in \mathit{C}^{2}( \mathbb{R}^3\backslash \overline{D}).$

\begin{remark}
The well-posedness of the medium scattering problems (\ref{eq:helmholtz})-(\ref{eq:boundarycondition}) and (\ref{LS})-(\ref{boundary}) can be established by the integral equation method (see, e.g.\cite{colton.kress, peter1998On }). The key ingredients for proving uniqueness are Green's theorem, Rellich's lemma and the unique continuation for elliptic equations. Existence of a solution is derived with the help of the representation theorem which leads to a Fredholm integral equation of Lippmann-Schwinger type for the displacement $u$ or $\mathbf{u}$. Under the weaker regularity assumption that $n, \rho\in L^\infty(\R^3)$, the well-posedness of the acoustic and elastic medium scattering problem can be established via variational approach together with properties of Dirichlet-to-Neumann operators for the Helmholtz and Naiver equations, we refer to \cite[Chapter 5.3]{colton.kress} and \cite{Hu} for the details.
\end{remark}

\subsection{Inverse problems and literature review}
Suppose that the contrast functions $|1-n|$ and $|1-\rho|$ is strictly positive in a neighborhood of $\partial D$ in $D$. In this paper we study the following two $\textit{inverse medium scattering problems}$ with infinitely many plane waves at a fixed energy:

{\bf IP1}: For a fixed $k$, determine the support $D$ of the inhomogeneous medium together with $n|_{\partial D}$ from knowledge of $u^{\infty}(\hat{x},d)$ for all $\hat{x},d\in \mathbb{S}^2 $.

{\bf IP2}: For a fixed $\omega$, determine the support $D$ of the inhomogeneous medium together with $\rho|_{\partial D}$ from knowledge of $\mathbf{u}^{\infty}(\hat{x}, d, q)$ for all $\hat{x},d, q\in \mathbb{S}^2,  d\cdot q=0$.

If $D$ is an impenetrable obstacle, many uniqueness results have been obtained  with one or many incident plane waves at a fixe energy. The first uniqueness result for sound-soft obstacles in acoustics was given in \cite{colton.kress} based on the ideas of Schiffer. If two scatterers produce the same far field pattern for all incident directions, Isakov \cite{Isakov1990On} derived a contradiction by considering a series of solutions with a singularity moving towards a point on the boundary of one obstacle but not contained in the other obstacle. Motivated by \cite{Isakov1990On}, Kirsch and Kress \cite{KK1993,colton.kress} simplified Isakov's method by using a priori estimate of classical solutions to an integral equation system that is equivalent to the scattering problem. This method also works for sound-hard and Robin-type impenetrable obstacles as well as penetrable ones satisfying interface transmission conditions with a discontinuity condition on the normal derivative (see \eqref{trans} below); see \cite{KK1993}. In linear elasticity it was applied in \cite{peter1993uniqueness} to get uniqueness in inverse time-harmonic scattering of elastic waves.

If $D$ is a penetrable body, H\"{a}hner in \cite{Peter2000penetrable} introduced a new technique to prove the unique determination of a penetrable, inhomogeneous and anisotropic body from a knowledge of the scattered near fields for all incident plane waves. H\"{a}hner's method is based on the existence, uniqueness and regularity of solutions to an interior transmission problem in $D$. The work \cite{Uhlmann} uses complex geometrical optics solutions of the Helmholtz equation to prove that a penetrable obstacle with $C^2$-smooth boundary in two dimensions can be uniquely reconstructed by near field measurements. The results of \cite{Uhlmann} was then extended to three dimensions in \cite{Yoshida} and to the case that $D$ is penetrable with Lipschitz continuous boundary \cite{Sini}. Moreover, \cite{MM2014} also uses complex geometrical optics solution of the elastic wave equation to get uniqueness with far-field measurements. There are also other works dealing with the penetrable body with embedded obstacles. It was proved in \cite{L-Z} that $D$ can be determined from knowledge of the acoustic far field pattern for incident plane waves when the refractive index function $n$ is a known constant in $D$. 
All the above uniqueness results were obtained on the jump condition for the normal derivatives, i.e.,
\be\label{trans}
\partial_{\nu} u^{+} =\gamma \partial_{\nu} u^{-} \text{ on } \partial D, \quad \gamma \neq 1.
\en

In this work we are interested in inverse medium scattering problems with the transmission coefficient $\gamma=1$ (see \eqref{trans}) both in acoustics and elastics. Physically, the {\em continuity} of the normal derivatives (in the sense of trace) leads to a weakly scattering effect by the penetrable obstacle, in comparing with the strongly scattering due to the transmission coefficient $\gamma\neq 1$. This perhaps explains why it is difficult to carry out the same a priori estimate as in \cite{KK1993} for singular solutions incited by the fundamental solution as an incoming wave. To overcome this difficulty, \cite{Yang2017Uniqueness} proposed a new method to establish uniqueness for recovering the shape of a penetrable $D$ in acoustic case. It is based on the $L^p$-estimate of hyper-singular solutions and also on the constructing a properly defined and well-posed interior transmission problem in a small domain inside $D$. The selection of a small domain ensures that the lowest transmission eigenvalue is large, so that a given wave number $k$ is not an eigenvalue of the constructed interior transmission problem. The same trick has been employed in \cite{xiang} for inverse elastic scattering problems. The main contribution of this paper is to propose a more direct approach by avoiding a priori estimate as much as possible and without using interior transmission eigenvalue problems. Due to the weakly scattering, it still seems necessary to use hyper-singular solutions. If the refractive index function admits jumps across the boundary, we employ singular total field caused by the first derivative of the fundamental solution when $\gamma=1$. This direct approach requires a more subtle analysis on the singular behavior of the scattered field at the boundary as the source position approaches to a boundary point from exterior, which has been used by Potthast \cite{Potthast} in acoustic and electromagnetic inhomogeneous medium scattering. In \cite{Potthast}, the unknown domain $D$ is assumed to be bounded, uniformly smooth and satisfy an exterior cone condition and the refractive index function is assumed to be H\"older continuous. Based on the lower and upper estimates of the scattered field generated by a multiple, Potthast \cite{Potthast} proved uniqueness for recovering the shape of $D$. The main contribution of this paper is to calculate the leading part of the asymptotics explicitly in terms of the medium parameters. Consequently, we show that the leading part yields information not only on the position and geometric shape of the scatterer but also on the value of the medium function at the boundary. Such an idea has been early used in \cite{EH11} for inverse electromagnetic scattering problems in a multi-layered medium in the TM polarization case where $\gamma\neq 1$ is the ratio of the two wavenumbers on both sides of the boundary.  

The rest of the paper is organized as follows. In Sections \ref{Acoustic}, we prove that far field patterns for incident plane waves with all incident directions at a fixed wave number uniquely determine the shape of the penetrable body $D$ in acoustics. In Section \ref{Elastic}, the hyper-singular source method will be carried over to inverse elastic scattering by penetrable bodies with the same uniqueness result.

\section{Inverse acoustic scattering problem}\label{Acoustic}

We briefly introduce the Lippmann-Schwinger equation and the volume potential. We first present a mapping property for the volume potential in Sobolev spaces \cite[Theorem 8.2]{colton.kress}.
\begin{theorem}\label{vp-acoustic}
Given two bounded domains $D$ and $G$, the volume potential
\begin{equation}
(V\varphi)(x):=\int_{D}\ \Phi(x,y)\varphi(y)\dif y,\ \ x\in  \mathbb{R}^3
\end{equation}
defines a bounded operator $V: L^2(D)\rightarrow H^2(G).$
\end{theorem}
Then the Lippmann-Schwinger equation is defined as
\begin{equation}
u(x)=u^{{\rm in}}(x)-k^2\int_{\mathbb{R}^3 }(1-n(y))\Phi(x,y)u(y)\dif y,\ \ x\in  \mathbb{R}^3 .
 \end{equation}
The medium scattering problem (\ref{eq:helmholtz})--(\ref{eq:boundarycondition}) is equivalent to the problem of solving the above integral equation; see \cite[Theorem 8.3]{colton.kress}.

Next we prove the unique determination of the support $D$ of the inhomogeneous medium from the far-field pattern $u^{\infty}(\hat{x}, d)$ for all $\hat{x}, d \in \mathbb{S}^{2}$ by using the Lippmann-Schwinger equation and first derivatives of the fundamental solution for the Helmholtz equation.
\begin{theorem}[Uniqueness with infinitely many plane waves] \label{thm:main}
Let $n\in L^\infty(\R^3)$ be real valued. Suppose that $|1 - n(x)| \geqslant C_0 > 0$ a.e. for $x \in D$ lying in a neighboring area of $\partial D$. Then the measurement data $u^{\infty}(\hat{x}, d)$ for all $\hat{x}, d \in \mathbb{S}^{2}$ uniquely determine $\partial D$. Furthermore, if $n\in C^{0,\gamma}(\overline{D})$ for some $\gamma\in(0,1)$, then the value of $n$ on $\partial D$, that is, $n|_{\partial D}$, can also be uniquely determined by the same measurement data.
\end{theorem}
\begin{remark}
We remark that the determination of $\partial D$ has been shown in \cite[Chapter 2]{Potthast} for H\"older continuous  refractive index functions. 
\end{remark}
\begin{proof}
Assume that there are two penetrable obstacles $D$ and $D'$ such that
\begin{equation}\label{eq:relation1}
u^{\infty}(\hat{x}, d) = u'^{\infty}(\hat{x}, d)\text{ for all } \hat{x}, d \in \mathbb{S}^{2}.
\end{equation}
If $D \neq D'$, we shall derive a contradiction from the above relation. By Rellich's lemma, $u^{{\rm sc}}(x,d)=u'^{{\rm sc}}(x, d)$ for all $ x \in G, d \in\mathbb{S}^{2}$, where $G$ denotes the unbounded component of $\mathbb{R}^{3}\backslash \overline{D \cup D'}$. Using the mixed reciprocity relation~\cite[Theorem 3.24]{colton.kress} one can deduce that
\begin{equation*}
4\pi u^{\infty}(-d,x)=4\pi u'^{\infty}(-d, x)\text{ for all }  d\in\mathbb{S}^{2},x\in G.
\end{equation*}
Rewriting the above equation with $-d = \hat{x}$, $x = z$ gives
	\begin{equation*}
		 u^{\infty}(\hat{x}, z) = u'^{\infty}(\hat{x},
		 z) \text{ for all }\hat{x} \in \mathbb{S}^{2}, \ z \in G.
	\end{equation*}
    Then we get
	\begin{equation*}
		u^{{\rm sc}}(x, z) = u'^{{\rm sc}}(x, z)\text{ for all } x, z \in G,
	\end{equation*}
due to the Rellich's lemma. This implies that
	\begin{equation} \label{eq:relation2}
		\nabla u^{{\rm sc}}\left(x, z\right) = \nabla u'^{{\rm sc}}\left(x, z\right)\text{ for all } x, z \in G.
	\end{equation}
	Here $u^{\infty}(\hat{x}, z)$ and $u^{{\rm sc}}(x, z)$ denote the far-field patterns
	and the scattered fields corresponding to a point source generated at $z \in
	\mathbb{R}^{3}$, respectively.
	
Choose $z^{*} \in \partial D \cap \partial G$ and $z^{*} \notin \overline{D'}$ (see Figure \refeq{fig:choiceOfzstar} for the choice of $z^{*}$) and let
$$z_{j}
	\coloneqq z^{*} + \frac{\nu(z^{*})}{j} \in \mathbb{R}^{3} \backslash
	\overline{D},\quad j = 1, 2, \cdots,$$
where $\nu(z^{*})$ is the unit outer normal at $z^{*}$. Define a sequence of incoming point source waves
 $$u_{j}^{{\rm in}}(x) = \Phi(x,
	z_{j}) := \frac{e^{ik\vert x - z_{j}\vert}}{4\pi \vert x - z_{j}\vert},\quad x\neq z_j.$$
	By straightforward computations, it holds that $u_{j}^{{\rm in}} \in L_{{\rm loc}}^{2}(\mathbb{R}^{3})$.
	\begin{figure}[htp]
		\centering
		\includegraphics{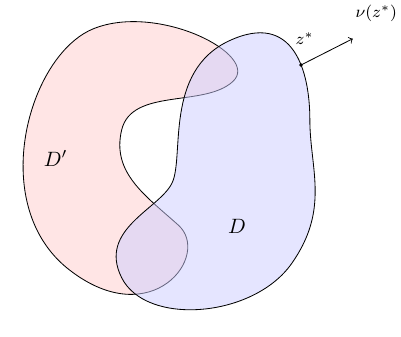}
		\caption{Illustration of the location of $z^{*}$ and the normal direction $\nu(z^{*})$.}
		\label{fig:choiceOfzstar}
	\end{figure}
	Recalling the Lippmann-Schwinger equation that is equivalent
	to the medium scattering problem, we get the integral equation
	\begin{equation} \label{eq:lippmannSchwinger}
		u(x, z_{j}) = u_{j}^{{\rm in}}(x) - k^{2}\int_{\mathbb{R}^{3}} (1 - n(y)) \Phi(x, y) u(y, z_{j})\dif y,\ x \in \mathbb{R}^{3}.
	\end{equation}
	Then we define the volume operator $K \colon L_{{\rm loc}}^{2}(\mathbb{R}^{3})
	\to H_{{\rm loc}}^{2}(\mathbb{R}^{3})$ by
	\begin{equation} \label{eq:operatorK}
		(K \varphi) (x) \coloneqq -k^{2}\int_{\mathbb{R}^{3}} (1 - n(y))\Phi(x, y)\varphi(y)\dif y, \ x \in \mathbb{R}^{3}
	\end{equation}
	where $\supp(1 - n) = \overline{D}$. With Theorem \ref{vp-acoustic}, it holds that $Ku_{j}^{{\rm in}} \in H_{{\rm loc}}^{2}
	(\mathbb{R}^{3})$ if $n \in L^{\infty}(\mathbb{R}^{3})$.
	 By the definition of $K$, we can reformulate (\ref{eq:lippmannSchwinger}) as an abstract operator equation
	\begin{equation*}
		(I - K)u(x, z_{j}) = u_{j}^{{\rm in}}(x),\ j = 1, 2, \cdots.
	\end{equation*}
	Write $u_{j}(x) \coloneqq u(x, z_{j})$ for notational simplicity and introduce the functions $w_{j} \coloneqq
	u_{j} - u_{j}^{{\rm in}}$. It is easy to see that the functions $w_{j}$
	satisfy the integral equation
	\begin{equation*}
		(I - K)w_{j} = Ku_{j}^{{\rm in}},\ j \geqslant 1.
	\end{equation*}

	Let $\Omega$ be a subdomain of $\mathbb{R}^{3}$ containing $D$ and $z^{*}$.
	The embedding of $H^{2}(\Omega)$ into $L^{2}(\Omega)$ implies that $I - K \colon
	L^{2}(\Omega) \to L^{2}(\Omega)$ is a Fredholm operator with the index zero. By the uniqueness of forward scattering, the integral equation
	$(I - K)w_{j} = 0$ has only the trivial solution $w_{j} = 0$. Then due to the
	Fredholm alternative, $(I - K)w_{j} = Ku_{j}^{{\rm in}}$ has a unique solution in $L^{2}(\Omega)$
	such that $\Vert w_{j}\Vert_{L^{2}(\Omega)} \leqslant C\Vert Ku_{j}^{{\rm in}}\Vert_{L^{2}(\Omega)}$.
	Hence
	\begin{align*}
		\Vert w_{j}\Vert_{H^{2}(\Omega)} &= \Vert Kw_{j} + Ku_{j}^{{\rm in}}\Vert_{H^{2}(\Omega)} \\
		&\leqslant \Vert Kw_{j}\Vert_{H^{2}(\Omega)} + \Vert Ku_{j}^{{\rm in}}\Vert_{H^{2}(\Omega)} \\
		&\leqslant C\Vert w_{j}\Vert_{L^{2}(\Omega)} + \Vert Ku_{j}^{{\rm in}}\Vert_{H^{2}(\Omega)} \\
		&\leqslant C\Vert Ku_{j}^{{\rm in}}\Vert_{H^{2}(\Omega)} + \Vert Ku_{j}^{{\rm in}}\Vert_{H^{2}(\Omega)} \\
		&\leqslant C\Vert Ku_{j}^{{\rm in}}\Vert_{H^{2}(\Omega)}.
	\end{align*}
	Moreover, by the boundedness of operator $K$,
	we have
	\begin{equation} \label{eq:uniform}
		\Vert w_{j}\Vert_{C(\Omega)} \leqslant C \Vert w_{j}\Vert_{H^{2}(\Omega)} \leqslant C \Vert Ku_{j}^{{\rm in}}\Vert_{H^{2}(\Omega)} \leqslant C \Vert u_{j}^{{\rm in}}\Vert_{L^{2}(\Omega)} \leqslant C.
	\end{equation}
	Note that the uniform boundness also follows from the solvability of the
	integral equation in $C(\Omega)$.
	
	It follows from the Lippmann-Schiwinger equation that
	\begin{equation*}
		w_{j}(x) = -k^{2}\int_{D} (1 - n(y))\Phi(x, y) (w_{j}(y) + u_{j}^{{\rm in}}(y))\dif y,\ x \in \mathbb{R}^{3}.
	\end{equation*}
To prove uniqueness of the inverse problem we shall take the first derivative of $w$ by
	setting $v_{j}(x) := \nabla w_{j}(x) \cdot \nu(z^{*})$. Then we can compute $v_{j}(x)$ as
	\begin{equation*}
		v_{j}(x) = -k^{2} \int_{D} (1 - n(y)) (\nabla_{x} \Phi(x, y) \cdot \nu(z^{*}))(w_{j}(y) + u_{j}^{{\rm in}}(y))\dif y.
	\end{equation*}
	Dividing $v_{j}(x)$ into three parts at $x = z_{j}$ yields
	\begin{align*}
		v_{j}(z_{j})= &-k^{2} \int_{D} (1 - n(y)) (\nabla_{x} \Phi(z_{j}, y) \cdot \nu(z^{*}))w_{j}(y)\dif y \\
		              &-k^{2} \int_{D \cap B_{\delta}(z^{*})} (1 - n(y)) (\nabla_{x} \Phi(z_{j}, y) \cdot \nu(z^{*}))u_{j}^{{\rm in}}(y)\dif y \\
		              &-k^{2} \int_{D \backslash B_{\delta}(z^{*})} (1 - n(y)) (\nabla_{x} \Phi(z_{j}, y) \cdot \nu(z^{*}))u_{j}^{{\rm in}}(y)\dif y\\
		            = &I_{1}(n,w_{j}) + I_{2}(n) + I_{3}(n),
	\end{align*}
	where $\delta > 0$ is sufficiently small.
Using (\ref{eq:uniform})
	and the definition of $u_{j}^{{\rm in}}\left(x\right)$,
we can estimate $I_1(n,w_{j})$, $I_2(n)$, $I_3(n)$ as follows:
	\begin{equation} \label{eq:i3i4}
		\vert I_{1}(n,w_{j})\vert + \vert I_{3}(n)\vert < M\quad \ \mbox{for all}\quad j \in \mathbb{N},
	\end{equation}
	and
	\begin{equation} \label{eq:i2}
		\begin{aligned}
			\vert I_{2}(n)\vert &= \bigg\vert -k^{2} \int_{D \cap B_{\delta}(z^{*})} (1 - n(y)) (\nabla_{x} \Phi(z_{j}, y) \cdot \nu(z^{*}))u_{j}^{{\rm in}}(y)\dif y\bigg\vert \\
			&= \bigg\vert k^{2}\int_{D \cap B_{\delta}(z^{*})} (1 - n(y)) \frac{e^{ik\vert z_{j} - y\vert}}{4\pi \vert z_{j} - y\vert} \frac{e^{ik\vert z_{j} - y\vert}}{4\pi \vert z_{j} - y\vert^{2}}\left(\frac{z_{j} - y}{\vert z_{j} - y\vert} \cdot \nu\left(z^{*}\right)\right)\dif y\bigg\vert + \mathcal{O}(1) \\
			&\geqslant k^{2}\int_{D \cap B_{\delta}(z^{*})} |1 - n(y)|\;\frac{\Re(e^{i2k\vert z_{j} - y\vert})}{16\pi^{2}\vert z_{j} - y\vert^{3}}\left(\frac{z_{j} - y}{\vert z_{j} - y\vert} \cdot \nu(z^{*})\right)\dif y + \mathcal{O}(1) \\
			&\geqslant \frac{C_0}{2}\frac{k^2}{16\pi^2}\int_{D \cap B_{\delta}(z^{*})} \frac{1}{\vert z_{j} - y\vert^{3}}\dif y + \mathcal{O}(1) \\
			&= C_0\frac{k^2}{16\pi} \ln(\delta j+1) + \mathcal{O}(1)
		\end{aligned}
	\end{equation}
	as $j \to \infty$. Note that the last step follows from  Lemma \ref{integ-computation}  below and we have used the condition that $|1-n|\geq C_0$ almost everywhere in $D \cap B_{\delta}(z^{*})$.
This together with (\ref{eq:i3i4}) proves that
	\begin{equation} \label{eq:vd1}
		\lim\limits_{j
			\to \infty} \vert v_{j}(z_{j})\vert = \infty.
	\end{equation}
	
	Analogously to the definition of $v_{j}$ for $D$, we can define $v'_{j}(
	x) = \nabla w'_{j}(x) \cdot \nu(z^{*})$, where
	$w'_{j} \coloneqq u'_{j} - u_{j}^{in}$ denotes the scattered field corresponding
	to another penetrable obstacle $D'$. Recalling that $z^{*} \in \partial D$ and $z^{*}
	\notin \overline{D'}$, we obtain form the well-posedness of the forward scattering
	problem for $D'$ that
	\begin{equation*}
	\vert \nabla w'_{j}(x)\vert + \vert w'_{j}\left(x\right)\vert \leqslant C\quad \text{ for all } \vert x - z_{j}\vert < \epsilon,\ \epsilon > 0.
	\end{equation*}
	This in particular implies that
	\begin{equation} \label{eq:vd2}
		\vert v'_{j}\left(z_{j}\right)\vert = \vert \nabla w'_{j}\left(z_{j}\right) \cdot \nu\left(z^{*}\right)\vert \leqslant C,\ j \to \infty.
	\end{equation}
	On the other hand, we have the relation
	\begin{equation} \label{eq:relation3}
		v_{j}\left(z_{j}\right) = v'_{j}\left(z_{j}\right) \text{ for all }\ j \geqslant 1,
	\end{equation}
which contradicts (\ref{eq:vd1}) and (\ref{eq:vd2}). This contradiction implies that $D = D'$.

Next, we are going to determine the value of $n(x)\in C^{0,\gamma}(\overline{D})$ on $\partial D$. Suppose there are two refractive index functions $n$ and $n'$ and assume on the contrary that $n(x)\neq n'(x)$ at some boundary point $x=z^*\in \partial D$.
By the continuity, we assume that $|n(x)-n'(x)|\geq c_{0}>0$ for all $x\in D \cap B_{\delta}(z^{*})$ for some $\delta>0$. Then choose a sequence of incoming point source waves $u_{j}^{{\rm in}}(x)$ with $z_{j}=z^{*}+\nu(z^{*})/{j}\in \mathbb{R}^{3} \backslash\overline{D}$ $(j = 1, 2, \cdots,)$. Following the uniqueness proof in determining the boundary, we obtain
\begin{equation*}
I_{1}(n,w_{j}) + I_{2}(n) + I_{3}(n)=I_{1}(n',w'_{j}) + I_{2}(n') + I_{3}(n'),
\end{equation*}
that is
\begin{equation} \label{eq1}
I_{2}(n)-I_{2}(n')=I_{1}(n',w'_{j})-I_{1}(n,w_{j})+I_{3}(n')-I_{3}(n).
\end{equation}
For all $j \in \mathbb{N}$, the right hand of \eqref{eq1} is bounded by
\begin{equation*}
|I_{1}(n',w'_{j})-I_{1}(n,w_{j})+I_{3}(n')-I_{3}(n)|\leq |I_{1}(n',w'_{j})|+|I_{1}(n,w_{j})|+|I_{3}(n')|+|I_{3}(n)|<M.
\end{equation*}
While for the left hand of \eqref{eq1}, the calculations in \eqref{eq:i2} show that
\begin{equation*}
\begin{aligned}
\vert I_{2}(n)-I_{2}(n')\vert &= \bigg\vert -k^{2} \int_{D \cap B_{\delta}(z^{*})} (n'(y)- n(y)) (\nabla_{x} \Phi(z_{j}, y) \cdot \nu(z^{*}))u_{j}^{{\rm in}}(y)\dif y\bigg\vert \\
			&\geqslant k^{2}\int_{D \cap B_{\delta}(z^{*})} |n'(y) - n(y)|\;\frac{\Re(e^{i2k\vert z_{j} - y\vert})}{16\pi^{2}\vert z_{j} - y\vert^{3}}\left(\frac{z_{j} - y}{\vert z_{j} - y\vert} \cdot \nu(z^{*})\right)\dif y + \mathcal{O}(1) \\
			&\geqslant \frac{c_0}{2}\frac{k^2}{16\pi^2}\int_{D \cap B_{\delta}(z^{*})} \frac{1}{\vert z_{j} - y\vert^{3}}\dif y + \mathcal{O}(1)\\
&= c_0\frac{k^2}{16\pi} \ln(\delta j+1) + \mathcal{O}(1)
		\end{aligned}
	\end{equation*}
as $j \to \infty$, i.e. $\lim\limits_{j\to\infty} \vert I_{2}(n)-I_{2}(n')\vert = \infty$. This is a contradiction which implies that $n(x)=n'(x)$ on $\partial D$.
\end{proof}

\begin{lemma} \label{integ-computation} Suppose that the boundary $\partial D$ is $C^2$-smooth and $\delta>0$ is sufficiently small.
A lower bound of the asymptotic behavior of $\vert I_{2}\vert$ can be estimated by
	\begin{equation*}
		\vert I_{2}\vert \geqslant C_0\frac{k^2}{16\pi} \ln(\delta j+1) + \mathcal{O}(1)\quad\mbox{as}\quad j\rightarrow\infty,
	\end{equation*}
	where $C_0>0$ is the positive lower bound of $|1-n|$ for $x\in D$ close to $\partial D$.
\end{lemma}

\begin{proof}
To compute the integral in $\vert I_{2}\vert$, we fix $z_j=z^*+\nu(z^*)/j$ for some $j\geq 1$ at the origin and assume $\nu(z^*)=(0,0,-1)$.
Introduce the cylindrical
	coordinate system $(r, \theta, z)$ for $y=(y^{(1)}, y^{(2)}, y^{(3)})\in D\cap B_\delta(z^*)$:
	\begin{equation*}
		\begin{cases}
			y^{(1)} &= r\cos\theta, \\
			y^{(2)} &= r\sin\theta,\\
			y^{(3)} &= z,
		\end{cases}
	\end{equation*}
	where $r \geqslant 0$, $0 < \theta \leqslant 2\pi$ and $z \in \mathbb{R}$. Then we know that
	\begin{equation*}
		\dif y = \dif y^{(1)}\dif y^{(2)}\dif y^{(3)} = \bigg\vert \frac{\partial(y^{(1)}, y^{(2)}, y^{(3)})}{\partial(r, \theta, z)}\bigg\vert \dif r\dif\theta \dif z = r \dif r\dif \theta \dif z.
	\end{equation*}
	Since $\delta$ is sufficiently small, we approximate the integral by integrating
	$1/\vert z_{j} - y\vert^{3}$ on a half ball  centered at $z^{*}$ under the
	cylindrical coordinate system (see Figure \ref{fig:integralI2}).  We observe
	that $r$ should satisfy $0 < r < \sqrt{\delta^{2}-(z - \frac{1}{j})^{2}}$ for fixed $z \in (\frac{1}{j}, \frac{1}{j}+\delta)$. Hence we have the following estimate
	
	\begin{figure}[htp]
		\centering
		\includegraphics{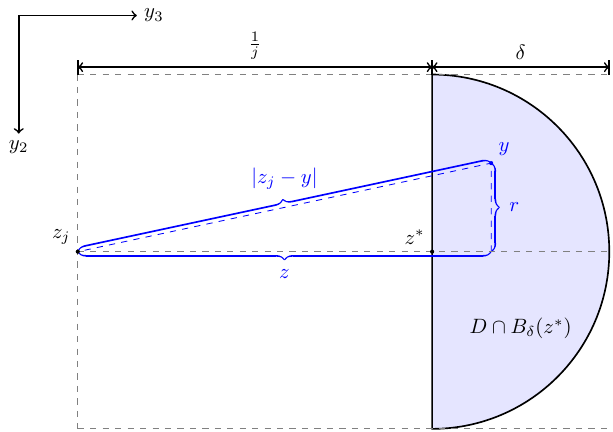}
		\caption{The cross section of domain $D \cap B_{\delta}(z^{*}) $ on the $oy_2y_3$-plane.}
		\label{fig:integralI2}
	\end{figure}
	
	\begin{align*}
		\int_{D \cap B_{\delta}(z^{*})} \frac{1}{\vert z_{j} - y\vert^{3}}\dif y &= \int_{\frac{1}{j}}^{\frac{1}{j}+\delta} \int_{0}^{2\pi} \int_{0}^{\sqrt{\delta^{2}-(z - \frac{1}{j})^{2}}}  (r^{2}+z^{2})^{-\frac{3}{2}} r \dif r \dif \theta \dif z + \mathcal{O}(1)\\
		                                                                         &=  2\pi\int_{\frac{1}{j}}^{\frac{1}{j}+\delta} \frac{1}{z} - \frac{1}{\sqrt{\delta^{2}-\frac{1}{j^{2}}+\frac{2}{j}z}} dz+ \mathcal{O}(1)\\
		                                                                         &= 2\pi\left(\ln{z}\Bigg\vert_{\frac{1}{j}}^{\frac{1}{j}+\delta} - j\sqrt{\delta^{2} - \frac{1}{j^{2}} + \frac{2}{j}z}\Bigg\vert_{\frac{1}{j}}^{\frac{1}{j}+\delta}\right) + \mathcal{O}(1)\\
		                                                                         &= 2\pi\ln{(\delta j+1)} - 4\pi\delta \frac{1}{\sqrt{\delta^{2} + \frac{1}{j^{2}}+\frac{2\delta}{j}} + \sqrt{\delta^{2} + \frac{1}{j^{2}}}} + \mathcal{O}(1)\\
		                                                                         &= 2\pi \ln{(\delta j+1)} + \mathcal{O}(1).
	\end{align*}
Now the desired estimate for the lower bound of $|I_2|$ follows from \eqref{eq:i2}.
\end{proof}


%

\begin{remark}\label{rem}
In Theorem \ref{thm:main}, we have assumed that the contrast function $1 - n$ is real valued (that is, the medium is non-absorbing) and does not change sign near $\partial D$. In the case of an absorbing medium (that is, $1 - n$ is complex valued) and $|\Re\left(1 - n\right)| \geqslant C > 0$ near $\partial D$, one can get the same uniqueness result. If the refractive index function is continuous at the interface but the absolute value of it's higher order derivatives keep a positive distance from zero near $\partial D$, it is still possible to get uniqueness by using singular solutions incited by higher order derivatives of the fundament solution. Moreover, we think that, through more delicate analysis, the hyper-singular source method developed here can be used to uniquely determine an analytic refractive index function defined on $\overline{D}$.  	
	\end{remark}

\section{Inverse elastic scattering problem}\label{Elastic}
 In this section, we first introduce the Lippmann-Schwinger equation for the Navier equation and properties of the fundamental tensor to the elastic wave equation. Then we apply the technique of using hyper-singular sources to obtain similar uniqueness
results for inverse elastic scattering from penetrable obstacles. Through out this paper, the symbol $\cdot$ stands for either the product between a matrix and a vector or the inner product between two vectors in $ \mathbb{R}^3$. All vectors are supposed to be column vectors.
\subsection{Fundamental solution and Lippmann-Schwinger equation}

We begin with the fundamental solution (or Green's tensor) $\Pi(x,y)$ to the Navier equation $\Delta^*\mathbf{u}+\omega^2\mathbf{u}=0$:
\begin{equation}
\Pi(x,y)=\dfrac{1}{\mu}\Phi_{k_s}(x,y)\mathbf{I}+\dfrac{1}{\omega^{2}}\nabla _x\nabla^{\top}_x \left[\Phi_{k_s}(x,y)-\Phi_{k_p}(x,y)\right],\ \ x\neq y
\end{equation}
where $ \Phi_k(x,y)=\dfrac{\mathrm{exp}(ik|x-y|)}{4\pi|x-y|}(k=k_p,k_s)$ is the fundamental solution to the Helmholtz equation.The symbol $\mathbf{I}$ stands for the $3\times3$  identity matrix and $'\top'$ denotes the transpose of a matrix or a vector.

Next, we will present some properties of the fundamental solution $\Pi(x,y)$, especially its behavior for $|x-y|\rightarrow 0.$ To this end we expand
 $\exp(ik|x|)/(4\pi|x|) $ to a power series to obtain
\begin{equation}\label{series}
\begin{aligned}
\frac{\exp(ik|x|)}{4\pi|x|}&=\frac{\cos(k|x|)}{4\pi|x|}+i\frac{\sin(k|x|)}{4\pi|x|}\\
                          &=\frac{1}{4\pi|x|}-\frac{k^2}{4\pi}|x|+k^4|x|^3f_1(k^2|x|^2)+ikf_2(k^2|x|^2)
\end{aligned}
\end{equation}
with two entire functions $f_1$ and $ f_2 $.

 Letting $r=|x-y|\rightarrow0$ and inserting the above expression (\ref{series}) into the definition of $\Pi(x,y)$, we obtain
\begin{equation}
\begin{aligned}
\Pi(x,y)&=\dfrac{1}{\mu}\Phi_{k_s}(x,y)\mathbf{I}+\dfrac{1}{\omega^{2}}\nabla _x\nabla^{\top}_x \left[\Phi_{k_s}(x,y)-\Phi_{k_p}(x,y)\right]\\
&=\frac{1}{\mu}\frac{1}{4\pi r}\mathbf{I}+\dfrac{1}{\omega^{2}}\nabla _x\nabla^{\top}_x \left(\frac{1}{8\pi}(k^2_p-k^2_s)|r|\right)+\mathcal{O}(1)\\
&=\frac{1}{\mu}\frac{1}{4\pi r}\mathbf{I}-\dfrac{\lambda+\mu}{8\pi\mu(\lambda+2\mu) }\nabla _x\nabla^{\top}_x (|r|)+\mathcal{O}(1).
\end{aligned}
\end{equation}
We denote the primary part of $\Pi(x,y)$ as
\begin{equation}
\Pi^0(x,y):=\frac{1}{\mu}\frac{1}{4\pi r}\mathbf{I}-\dfrac{\lambda+\mu}{8\pi\mu(\lambda+2\mu) }\nabla _x\nabla^{\top}_x (|r|)
\end{equation}
which is known as the Kelvin's matrix\cite{Kupradze}.
Direct calculations show that
\begin{equation}
\begin{aligned}
 \nabla_x r=(x-y)/r,\ \
 \nabla _x\nabla^{\top}_x (|r|)=\frac{1}{r}\mathbf{I}-\frac{1}{r}\hat{(x-y)}\otimes\hat{(x-y)}
\end{aligned}
\end{equation}
where the symbol $ '\otimes '$ represents the tensor product $a\otimes a:=aa^{\top}$ for the vector vector $a\in \R^{3}$, and $ \hat{(x-y)}:=(x-y)/|x-y|.$ Set $\alpha:=\frac{\lambda+3\mu}{8\pi\mu(\lambda+2\mu)}$ and $ \beta:=\frac{\lambda+\mu}{8\pi\mu(\lambda+2\mu)} $. Then
\begin{equation}
\Pi^0(x,y)=\alpha\frac{1}{r}\mathbf{I}+\beta\frac{1}{r}\hat{(x-y)}\otimes\hat{(x-y)}
\end{equation}
and
\begin{equation}\label{Asymptotic}
\Pi(x,y)=\Pi^0(x,y)+\mathcal{O}(1),\ \ \ |x-y|\rightarrow 0£¬
\end{equation}
which means that $\Pi(x,y)$ and $\Pi^0(x,y)$ have the same singularity of type $ 1/|x-y|$.

\begin{lemma}\label{cal1}
Let $ b\in  \mathbb{R}^3$ be a vector. Then it holds that

(1)\begin{equation}\label{singularity}
\begin{aligned}
\nabla_x\left(\Pi^0(x,y)\cdot  b \right)=\frac{1}{r^2}\left[-\alpha\left( b\otimes \overline{x}\right)+\beta\left((\overline{x}\cdot  b)\mathbf{I}+ \overline{x}\otimes\nu-3(\overline{x}\cdot  b)\overline{x}\otimes \overline{x}\right)\right].
\end{aligned}
\end{equation}

 (2)\begin{equation}
\begin{aligned}
&\nabla_x\left(\Pi^0(x,y)\cdot  b\right)\cdot \left(\Pi^0(x,y)\cdot  b\right)\\
=&\frac{1}{r^3}\left[-\alpha^2(\overline{x}\cdot b)b+\left(\alpha\beta-(3\alpha\beta+\beta^2)(\overline{x}\cdot b)^2\right)\overline{x}\right].
\end{aligned}
\end{equation}
Here $r=|x-y|$ and $\overline{x}:=\hat{(x-y)}.$

(3) In the special case that $b=\hat{x-y} $, we have \begin{equation}
\begin{aligned}
&\nabla_x\left(\Pi^0(x,y )\cdot  b\right)\cdot \left(\Pi^0(x,y)\cdot  b\right)=\frac{-1}{r^3}(\alpha+\beta)^2\overline{x}.
\end{aligned}
\end{equation}

\end{lemma}
\begin{proof}
(1)
Set $z:=x-y=(z^{(1)}, z^{(2)}, z^{(3)})^{\top}$ and $\mathbf{A}:=\nabla_x\left(\frac{1}{r}\hat{z}\otimes\hat{z}\cdot  b \right)\in \mathbb{R}^{3\times3}$.  Since $ \hat{z}\otimes \hat{z}\cdot b=(\hat{z}\cdot b)\hat{z} $, the $j-$th row of $\mathbf{A}$ takes the form
\begin{equation}
\begin{aligned}
\mathbf{A}_j&=\nabla_x\left(\frac{z^{(j)}}{r^3} (z\cdot b) \right)\\
            &=(z\cdot b)\nabla_x(\frac{z^{(j)}}{r^3})+\frac{z^{(j)}}{r^3}b^{\top}\\
            &=(z\cdot b) \frac{1}{r^3}e_j^{\top}-(z\cdot b)\frac{3z^{(j)}}{r^5}z^{\top}+\frac{z^{(j)}}{r^3}b^{\top}.
\end{aligned}
\end{equation}
Here $\left\lbrace e_j\right\rbrace_{j=1}^{3}$ denote the unit vectors in $ \mathbb{R}^3$ in the Cartesian coordinates. From the expression of $\mathbf{A}_j$  we can derive
\begin{equation}\label{matrix}
 \begin{aligned}
 \mathbf{A}= \frac{1}{r^3}(z\cdot b)\mathbf{I}+\frac{1}{r^3}z\otimes b-\frac{3}{r^5}(z\cdot b)z\otimes z.
 \end{aligned}
 \end{equation}
Substituting (\ref{matrix}) into $\Pi^0(x,y)$ gives
\begin{equation}
\begin{aligned}
\nabla_x\left(\Pi^0(x,y)\cdot  b\right)=-\frac{\alpha}{r^3}b\otimes z+\beta\left[\frac{1}{r^3}(z\cdot b)\mathbf{I}+\frac{1}{r^3}z\otimes b-\frac{3}{r^5}(z\cdot b)z\otimes z\right].
\end{aligned}
\end{equation}
The proof is now completed by replacing $z$ with $x-y$.

(2) Recalling \eqref{singularity}, we know
\begin{equation}\label{mul1}
\begin{aligned}
\nabla_x\left(\Pi^0(x, y)\cdot  b\right) &=\frac{1}{r^2}\left[-\alpha\left(b\otimes \overline{x}\right)+\beta\left((\overline{x}\cdot b)\mathbf{I}+ \overline{x}\otimes b-3(\overline{x}\cdot b)\overline{x}\otimes\overline{x}\right)\right]
\end{aligned}
\end{equation}
and
\begin{equation}\label{mul2}
\begin{aligned}
 \Pi^0(x,y )\cdot  b =\frac{1}{r}[\alpha\mathbf{I}+\beta\overline{x}\otimes\overline{x}]\cdot b
\end{aligned}
\end{equation}
Now, taking the product between (\ref{mul1}) and (\ref{mul2}) yields
\begin{equation}
\begin{aligned}
&\nabla_x\left(\Pi^0(x,y )\cdot  b\right)   \cdot \left(\Pi^0(x,y )\cdot  b\right) \\
=&\frac{1}{r^3}\left[-\alpha(\alpha+\beta)\left(b\otimes \overline{x}\right)\cdot b-\beta^2(\overline{x}\cdot b)(\overline{x}\otimes\overline{x})\cdot b\right]\\
+&\frac{\alpha\beta}{r^3}\left[(\overline{x}\cdot b)b+\overline{x}-3(\overline{x}\cdot b)(\overline{x}\otimes\overline{x})\cdot b\right].
\end{aligned}
\end{equation}
The proof is finished by applying identities
\begin{equation*}\label{iden}
 \begin{aligned}
& \left(\nu\otimes \overline{x}\right)\cdot(\overline{x}\otimes\overline{x})\cdot\nu=\nu\cdot\overline{x}^{\top}\cdot\overline{x}\cdot\overline{x}^{\top}\cdot\nu=\left(\nu\otimes \overline{x}\right)\cdot\nu,\\
& \left(\overline{x}\otimes \nu\right)\cdot(\overline{x}\otimes\overline{x})\cdot\nu=
 (\overline{x}\cdot \nu)(\overline{x}\otimes\overline{x})\cdot\nu,\\
& \left(\nu\otimes \overline{x}\right)\cdot\nu=(\overline{x}\cdot \nu)\nu,\\
&(\overline{x}\otimes\overline{x})\cdot\nu=(\overline{x}\cdot \nu)\overline{x}.
 \end{aligned}
 \end{equation*}

(3) If $b=\hat{x-y}$, then using the results of the second assertion we have
\begin{equation}
\begin{aligned}
&\nabla_x\left(\Pi^0(x,y)\cdot  b\right)\cdot \left(\Pi^0(x,y)\cdot  b\right)\\
=&\frac{1}{r^3}\left[-\alpha^2(\overline{x}\cdot b)b+\left(\alpha\beta-(3\alpha\beta+\beta^2)(\overline{x}\cdot b)^2\right)\overline{x}\right]\\
=&\frac{1}{r^3}\left[-\alpha^2 b+\left(\alpha\beta-(3\alpha\beta+\beta^2)\right)\overline{x}\right]\\
=&\frac{-1}{r^3}(\alpha+\beta)^2\overline{x}.
\end{aligned}
\end{equation}
\end{proof}

At last, we  introduce the Lippmann-Schwinger equation that is equivalent to the scattering problem (\ref{LS})-(\ref{rd}). For this purpose we need to consider the volume potential
\begin{equation}
(V\varphi)(x):=\int_{D}\Pi(x,y)\varphi(y)\dif y,\ \ x\in  \mathbb{R}^3.
\end{equation}
Below we show a classical mapping property property of $V$ in H\"older spaces, the proof of which can be found in \cite[Theorem 5.6]{peter1998On}.
\begin{theorem}\label{TH3.2}
Let $D \subset B_R $, where $B_R $ is a ball centered at origin with radius R. If $\varphi\in [\mathit{C}^{0,\gamma}(D)]^{3}$ with $\gamma\in(0,1)$, then the volume potential $ V\varphi\in [\mathit{C}^2(B_R)]^{3}$.
\end{theorem}

By using Lax's Theorem\cite[Theorem 3.5]{colton.kress}, one can deduce a mapping property for the volume potential in Sobolev spaces based on the result of Theorem \ref{TH3.2}.

\begin{lemma}\label{volume-potential}
The volume potential defines a bounded operator $ V: [L^2(D)]^{3}\rightarrow [H^2(B_R)]^{3}$.
\end{lemma}
The Lippmann-Schwinger equation is defined as
\begin{equation}
\mathbf{u}(x)=\mathbf{u}^{{\rm in}}(x)-\omega^2\int_{\mathbb{R}^3 }(1-\rho(y))\Pi(x,y)\mathbf{u}(y)\dif y,\ \ x\in  \mathbb{R}^3 .
 \end{equation}

In was proved in \cite[Lemma 5.7]{peter1998On} that a solution of the above Lippmann-Schwinger integral equation is a solution to the scattering problem (\ref{LS})-(\ref{rd}) and vice versa. Then for $\rho\in [L^{\infty}(\mathbb{R}^{3})]^{3}$ satisfying $\rm{supp} (1-\rho)=\overline{D}$, it's easy to show that the scattering problem (\ref{LS})--(\ref{boundary}) has a unique solution $\mathbf{u}^{sc}\in  [H_{{\rm loc}}^{2}(\mathbb{R}^{3})]^{3}.$

\subsection{Uniqueness for inverse elastic scattering}
Given an incident plane wave $ \mathbf{u}^{{\rm in}}(x,d,q) $, we will use the notations $ \mathbf{u}^{{\rm sc}}(x,d,q)$, $ u(x,d,q)$ and $ \mathbf{u}^{\infty}(\hat{x},d,q)$ to
indicate dependence of the scattered field, of the total field and of the far field pattern on the incident direction $d$ and polarization $q$, respectively.  For an incoming point source $\mathbf{w}^{{\rm in}}(x,y,a)=\Pi(x,y)\cdot a$ located at $y\in D^c $ with the unit vector $a\in \C^3$, we denote the scattered field by $ \mathbf{w}^{{\rm sc}}(x,y,a)$, the total field by $ \mathbf{w}(x,y,a)$, and the far field pattern of the scattered field by $ \mathbf{w}^{\infty}(\hat{x},y,a)$. We refer to  \cite[Theorem 2.1]{xiang} for a mixed reciprocity relation that is analogous to the acoustic case as follows.
\begin{lemma}
For any $y\in D^c $ and a fixed unit vector $a$, we have
\begin{equation}
q\cdot \mathbf{w}^{\infty}(-d,y,a)+ d\cdot \mathbf{w}^{\infty}(-d,y,a)=a\cdot \mathbf{u}^{{\rm sc}}(y,d,q)
\end{equation}
for all $ d,q\in\mathbb{S}^2, d\cdot q=0. $
\end{lemma}

Now we state and prove the uniqueness result in recovering the shape of the inhomogeneous penetrable obstacle $D$ in linear elasticity.
\begin{theorem}\label{main2}
Suppose $\rho\in [L^{\infty}(\mathbb{R}^{3})]^{3}$ is a real-valued function such that $|1-\rho|\geq C $ for some  constant $C>0$ and for $x \in D$ lying in a neighboring area of $\partial D$. Then the measurement data $ \mathbf{u}^{\infty}(\hat{x},d,q)$ for all $\hat{x},d, q\in \mathbb{S}^2 $ can uniquely determine $\partial D$. Furthermore, if $\rho(x)\in [C^{0,\gamma}(\overline{D})]^{3}$ for some $\gamma\in(0,1)$, then the value of $\rho(x)$ on $\partial D$ can be also uniquely determined by the same measurement data.
\end{theorem}
\begin{proof}
Assume that  $(D,\rho)$ and $(\tilde{D},
\tilde{\rho})$ generate the same far field patterns, that is,
\begin{equation}
\mathbf{u}^{\infty}(\hat{x},d,q)=\tilde{\mathbf{u}}^{\infty}(\hat{x},d,q)\ \ \text{ for all } \hat{x}, d ,q\in \mathbb{S}^{2},\ d\cdot q=0.
\end{equation}
Then by using Rellich's lemma, we obtain
\begin{equation}
\mathbf{u}^{{\rm sc}}(x,d,q)=\tilde{\mathbf{u}}^{{\rm sc}}(x,d,q)\ \ \text{ for all } x\in G,\ d ,q\in \mathbb{S}^{2},\ d\cdot q=0,
\end{equation}
where $G$ is denoted as the unbounded connected domain of  $\mathbb{R}^{3}\backslash \overline{D \cup \tilde{D}}$. Applying the above mixed reciprocity relation theorem,
we can deduce that for any fixed vector $a$
\begin{equation}
\mathbf{w}^{\infty}(-d,x,a)=\tilde{\mathbf{w}}^{\infty}(-d,x,a)\ \ \text{ for all } x\in G, d \in \mathbb{S}^{2}.
\end{equation}
Using Rellich's lemma again, we get
\begin{equation}\label{equality}
\mathbf{w}^{{\rm sc}}(y,x,a)=\tilde{\mathbf{w}}^{{\rm sc}}(y,x,a)\ \ \text{ for all } x,y\in G.
\end{equation}	
	
If $D\neq \tilde{D}$, we suppose without loss of generality that there exists a point $ z^{*} $ satisfying $z^{*} \in \partial D\cap \partial G $ and $z^{*} \notin \tilde{D}$. As done in the acoustic case, we set
\begin{equation*}
z_{j} = z^{*} + \frac{\nu(z^{*})}{j},\ \ j=1,2,\cdots
\end{equation*}
where $\nu(z^{*})$ is the unit normal at $z^{*}$ directed into the exterior of $D$. Apparently, $z_{j}\in G $ for $j\geq j_0$. Next, we consider the scattering problems due to point sources $\mathbf{w}_j^{{\rm in}}(x,z_{j},\nu(z^{*}))=\Pi(x,z_{j})\cdot \nu(z^{*}) $. It's easy to obtain $w_j^{{\rm in}}\in [L_{{\rm loc}}^{2}(\mathbb{R}^{3})]^{3}$ from the singularity of $\Pi(x,y).$ The total field $\mathbf{w}(x,z_{j},\nu(z^{*}))$ can be given by the Lippmann-Schwinger equation\cite{peter1998On}, that is
\begin{equation}
\mathbf{w}(x,z_{j},\nu(z^{*})= w_j^{{\rm in}}(x,z_{j},\nu(z^{*})) -\omega^2\int_{\mathbb{R}^3 }(1-\rho(y))\Pi(x,y) \mathbf{w}\left(y,z_{j},\nu(z^{*})\right)\dif y,\ \ x\in  \mathbb{R}^3 .
\end{equation}
Then
\begin{equation}\label{sch-win}
\mathbf{w}_j^{{\rm sc}}(x,z_{j},\nu(z^{*}))=-\omega^2\int_{\mathbb{R}^3 }(1-\rho(y)) \Pi(x,y)\left(\mathbf{w}_j^{{\rm in}}(x,z_{j},\nu(z^{*}))+\mathbf{w}_j^{{\rm sc}}(x,z_{j},\nu(z^{*}))\right)\dif y,\ \ x\in  \mathbb{R}^3 .
\end{equation}
Define the volume potential $V:[L_{{\rm loc}}^{2}(\mathbb{R}^{3})]^{3} \to [H_{{\rm loc}}^{2}(\mathbb{R}^{3})]^{3}$
\begin{equation}
 (V\varphi)(x)=-\omega^2\int_{\mathbb{R}^3 }(1-\rho(y))\Pi(x,y)\varphi(y)\dif y,\ \ x\in  \mathbb{R}^3.
\end{equation}
 We can rewrite the equation (\ref{sch-win}) as
 \begin{equation}
 (I-V)\mathbf{w}_j^{{\rm sc}}(x,z_{j},\nu(z^{*}))=V\mathbf{w}_j^{{\rm in}}(x,z_{j},\nu(z^{*})), \ \ x\in  \mathbb{R}^3, j=1,2,\cdots .
\end{equation}
Together with the facts $\supp(1 - \rho)=\overline{D}\subset B_R$, $\rho(x)\in [L^{\infty}(\mathbb{R}^{3})]^{3}$ and Lemma \ref{volume-potential}, we get a similar  estimate to (\ref{eq:uniform}) as
\begin{equation} \label{uniform-el}
\Vert \mathbf{w}_j^{{\rm sc}}\Vert_{[C(B_R)]^{3}} \leqslant C_1 \Vert \mathbf{w}_j^{{\rm sc}}\Vert_{[H^{2}(B_R)]^{3}} \leqslant C_2 \Vert V \mathbf{w}_j^{{\rm in}}\Vert_{[H^{2}(B_R)]^{3}} \leqslant C_3 \Vert \mathbf{w}_j^{{\rm in}}\Vert_{[L^{2}(B_R)]^{3}} \leqslant C_4
\end{equation}
where $B_R$ is a ball with radius $R$ containing $D$ and $z^{*}$.

Define $\bm{\upsilon}_j(x):=\nabla_x[\mathbf{w}_j^{{\rm sc}}(x,z_{j},q)\cdot \nu(z^{*})]$. Then,
\begin{equation}\label{main-eq}
\begin{aligned}
\bm{\upsilon}_j(x)=-\omega^2\int_{D}(1-\rho(y))\nabla_x[\Pi(x,y)\cdot \nu(z^{*})]\left(\mathbf{w}_j^{{\rm in}}(y,z_{j},\nu(z^{*})) +\mathbf{w}_j^{{\rm sc}}(y,z_{j},\nu(z^{*}))\right)\dif y.
\end{aligned}
\end{equation}
We continue to calculate $\bm{\upsilon}_j(x)$ at $x=z_j$ by
\begin{equation}\label{main-eq-devide}
\begin{aligned}
 \bm{\upsilon}_j(z_j)
&=-\omega^2\int_{D }(1-\rho(y))\nabla_x[\Pi(x,y)\cdot \nu(z^{*})]|_{x=z_j}\mathbf{w}_j^{{\rm sc}}(y,z_{j},\nu(z^{*}))\dif y\\
&  -\omega^2\int_{{D\cap B_{\delta}(z^{*})}}(1-\rho(y))\nabla_x[\Pi(x,y)\cdot \nu(z^{*})]|_{x=z_j}\mathbf{w}_j^{{\rm in}}(y,z_{j},\nu(z^{*}))\dif y\\
&  -\omega^2\int_{{D\setminus B_{\delta}(z^{*})}}(1-\rho(y))\nabla_x[\Pi(x,y)\cdot \nu(z^{*})]|_{x=z_j}\mathbf{w}_j^{{\rm in}}(y,z_{j},\nu(z^{*}))\dif y\\
&:=\mathit{I}_1(\rho,\mathbf{w}_j^{{\rm sc}})+\mathit{I}_2(\rho)+\mathit{I}_3(\rho),\ \ x\in  \mathbb{R}^3 ,
\end{aligned}
\end{equation}
where $B_{\delta}(z^{*})$ is a small ball centered at $z^{*}$ with a small radius $\delta>0$.

Combining equation (\ref{Asymptotic}) and Lemma \ref{cal1} we conclude that $\nabla_x[\Pi(x,y)\cdot \nu(z^{*})]$ has a singularity of type $1/|x-y|^2$. This together with
the fact that $\mathbf{w}_j^{{\rm in}}\in [L_{{\rm loc}}^{2}(\mathbb{R}^{3})]^{3}$ and (\ref{uniform-el}) gives
\begin{equation}
|\nu(z^{*})\cdot\mathit{I}_1(\rho,\mathbf{w}_j^{{\rm sc}})|+|\nu(z^{*})\cdot\mathit{I}_3(\rho)|\leq M,
\end{equation}
where $M$ is a positive constant independent of $j\in \mathbb{N}.$

From (\ref{Asymptotic}) it follows that $\mathbf{w}_j^{{\rm in}}(y,z_{j},\nu(z^{*})) =\Pi^0(y,z_{j})\cdot \nu(z^{*})+ \mathcal{O}(1)$. Hence,
\begin{equation} \label{eqq1}
\begin{aligned}
\mathit{I}_2(\rho)= -\omega^2\int_{{D\cap B_{\delta}(z^{*})}}(1-\rho(y))\nabla_x[\Pi^0(x,y)\cdot \nu(z^{*})]|_{x=z_j}\cdot \left(\Pi^0(y,z_j)\cdot \nu(z^{*}) \right) \dif y + \mathcal{O}(1).
\end{aligned}
\end{equation}
Denote $\overline{z}=(z_j-y)/|(z_j-y)|$. Using Lemma \ref{cal1} one can show
\begin{equation} 
\begin{aligned}
&|\nu(z^{*})\cdot\mathit{I}_2(\rho)|\\
=&\left| -\omega^2\int_{{D\cap B_{\delta}(z^{*})}}(1-\rho(y))\left[\nabla_x[\Pi^0(x,y)\cdot \nu(z^{*})]|_{x=z_j}\cdot \left(\Pi^0(y,z_j)\cdot \nu(z^{*}) \right)\right]\cdot\nu(z^{*})  \dif y\right| + \mathcal{O}(1)\\
=&\omega^2\left|\int_{{D\cap B_{\delta}(z^{*})}}\frac{(1-\rho(y))}{|z_j-y|^3}\left[-\alpha^2(\overline{z}\cdot\nu(z^{*}))+\left(\alpha\beta-(3\alpha\beta+\beta^2)(\overline{z}\cdot\nu(z^{*}))^2\right)\overline{z}\cdot \nu(z^{*})\right]\dif y\right| + \mathcal{O}(1).
\end{aligned}
\end{equation}
Since $\overline{z}\cdot \nu(z^{*})\rightarrow 1$ as $j\rightarrow \infty$, it is easy to see
\begin{equation*}
\left[-\alpha^2(\overline{z}\cdot\nu(z^{*}))+\left(\alpha\beta-(3\alpha\beta+\beta^2)(\overline{z}\cdot\nu(z^{*}))^2\right)\overline{z}\cdot \nu(z^{*})\right]\rightarrow -\left(\alpha+\beta\right)^2,\ \ j\rightarrow \infty.
\end{equation*}
Recalling that $\alpha=\frac{\lambda+3\mu}{8\pi\mu(\lambda+2\mu)}$ and $ \beta=\frac{\lambda+\mu}{8\pi\mu(\lambda+2\mu)} $, we get $\alpha+\beta=\frac{1}{4\pi\mu}$. Then
\begin{equation} \label{constant}
\begin{aligned}
&|\nu(z^{*})\cdot\mathit{I}_2(\rho)|\\
\geq &C\omega^2\left|\int_{{D\cap B_{\delta}(z^{*})}}\frac{1}{|z_j-y|^3}\left[-\left(\alpha+\beta\right)^2\right]\dif y\right| + \mathcal{O}(1)\\
\geq &\frac{C\omega^2}{16\pi^2\mu^2} \left|\int_{{D\cap B_{\delta}(z^{*})}}\frac{1}{|z_j-y|^3}\dif y\right| + \mathcal{O}(1)\\
=&\frac{C\omega^2}{8\pi\mu^2} \ln j+ \mathcal{O}(1),\ \ j\rightarrow \infty.
\end{aligned}.
\end{equation}
Here $C$ means the lower bound of the contrast function $|1-\rho|$. This implies that $|\nu(z^{*})\cdot\bm{\upsilon}_j(z_j)|\rightarrow \infty$ as $j\rightarrow \infty $.

On the other hand, we can define $\tilde{\bm{\upsilon}}_j(x):=\nabla_x[\tilde{\mathbf{w}}_j^{{\rm sc}}(x,z_{j},\nu(z^{*})) \cdot \nu(z^{*})]$, where $\tilde{\mathbf{w}}_j^{{\rm sc}}(x,z_{j},\nu(z^{*}))$ is the scattered field corresponding to $\tilde{D}$ and the incident point source $ \mathbf{w}_j^{{\rm in}}(x,z_{j},\nu(z^{*})).$ Since $z^{*} \in \partial D$ and $z^{*} \notin \tilde{D}$, it follows from the well-posedness of the direct scattering problem that
\begin{equation}
\begin{aligned}
|\nu(z^{*})\cdot\nabla_x[\tilde{\mathbf{w}}_j^{{\rm sc}}(x,z_{j},\nu(z^{*}))\cdot \nu(z^{*})]|+|\nu(z^{*})\cdot\tilde{\mathbf{w}}_j^{{\rm sc}}(x,z_{j},\nu(z^{*}))\cdot \nu(z^{*})|\leq M, \ \ x\in B_\delta(z_j).
\end{aligned}
\end{equation}
Here $M$ is a constant and $B_\delta(z_j)$ is a ball centered at $z_j$ outside of $\tilde{D}.$ The above inequality implies that $|\nu(z^{*})\cdot\tilde{\bm{\upsilon}}_j(z_j)|\leq M $ as $j\rightarrow \infty $.
Recalling from the relation (\ref{equality}) that $\bm{\upsilon}_j(z_j)=\tilde{\bm{\upsilon}}_j(z_j)$ for all $j\in  \mathbb{N}$, we get a contradiction here, which finishes the proof of $D=\tilde{D}$.

Next, we are going to determine the value of $\rho(x)\in[C^{0,\gamma}(\overline{D})]^{3}$ on $\partial D$. Assume on the contrary that $\rho(z^*)\neq \tilde{\rho}(z^*)$ at some boundary point $z^*\in\partial D$. We assume futher that $|\rho(z^*)-\tilde{\rho}(z^*)|\geq c>0$ for $x\in D \cap B_{\delta}(z^{*})$ and with some $\delta>0$. Then choose a sequence of incoming point source waves $u_{j}^{{\rm in}}(x)$ with $z_{j}=z^{*}+\nu(z^{*})/{j}\in \mathbb{R}^{3} \backslash\overline{D}$ $(j = 1, 2, \cdots,)$. Following the above processes, we obtain that
\begin{equation*}
I_{1}(\rho,\mathbf{w}_j^{{\rm sc}}) + I_{2}(\rho) + I_{3}(\rho)=I_{1}(\tilde{\rho},\tilde{\mathbf{w}}_j^{{\rm sc}}) + I_{2}(\tilde{\rho}) + I_{3}(\tilde{\rho}),
\end{equation*}
that is,
\begin{equation} \label{eq2}
\nu(z^*)\cdot [I_{2}(\rho)-I_{2}(\tilde{\rho})]=\nu(z^*)\cdot [I_{1}(\tilde{\rho},\tilde{\mathbf{w}}_j^{{\rm sc}})-I_{1}(\rho,\mathbf{w}_j^{{\rm sc}}) +I_{3}(\tilde{\rho})-I_{3}(\rho)].
\end{equation}
The right hand of \eqref{eq2} can be bounded by
\begin{equation*}
|\nu(z^{*})\cdot I_{1}(\tilde{\rho},\tilde{\mathbf{w}}_j^{{\rm sc}})|+|\nu(z^{*})\cdot I_{1}(\rho,\mathbf{w}_j^{{\rm sc}})|+|\nu(z^{*})\cdot I_{3}(\tilde{\rho})|+|\nu(z^{*})\cdot I_{3}(\rho)|<M, \quad  \forall j\in\mathbb{N}.
\end{equation*}
While for the left hand of \eqref{eq2}, repeating the calculations in \eqref{eqq1}--\eqref{constant} one deduces that
\begin{equation*}
\begin{aligned}
&\vert \nu(z^{*})\cdot I_{2}(\rho)-\nu(z^{*})\cdot I_{2}(\tilde{\rho})\vert  \\
=&\left|-\omega^2\int_{{D\cap B_{\delta}(z^{*})}}(\tilde{\rho}(y)-\rho(y)) \nabla_x[\Pi^0(x,y)\cdot \nu(z^{*})]|_{x=z_j}\cdot \left(\Pi^0(y,z_j)\cdot \nu(z^{*}) \right) \dif y\right| + \mathcal{O}(1) \\
\geq & \frac{c\,\omega^2}{16\pi^2\mu^2} \left|\int_{{D\cap B_{\delta}(z^{*})}}\frac{1}{|z_j-y|^3}\dif y\right| + \mathcal{O}(1) =\frac{c\, \omega^2}{8\pi\mu^2} \ln j+ \mathcal{O}(1),
		\end{aligned}
	\end{equation*}
as $j \to \infty$, i.e. $\lim\limits_{j\to\infty} \vert \nu(z^{*})\cdot I_{2}(\rho)-\nu(z^{*})\cdot I_{2}(\tilde{\rho})\vert = \infty$. This contradiction implies that $\rho(x)=\tilde{\rho}(x)$ on $\partial D$.
\end{proof}

%

\section{Acknowledgments}

The work of G. Hu is partially supported by the National Natural Science Foundation of China (No. 12071236) and the Fundamental Research Funds for Central Universities in China (No. 63233071). The work of J. Xiang is supported by the Natural Science Foundation of China (No. 12301542), the Natural Science Foundation of Hubei (No. 2022CFB725) and the Natural Science Foundation of Yichang (No. A23-2-027).

\end{document}